\documentclass[10pt,leqno,a4paper]{article}

\usepackage[english]{babel}
\usepackage{amsmath}
\usepackage{amsfonts}
\usepackage{amsthm}
\usepackage{amssymb}

\usepackage{xspace}
\usepackage{euscript}
\usepackage{graphicx}
\usepackage{amscd}
\usepackage{tabularx}

\usepackage{enumerate}
 \usepackage{epsfig} 
 \usepackage{graphics} 
\theoremstyle{plain}
\newtheorem{theo}{Theorem}[section]
\newtheorem{prop}[theo]{Proposition}
\theoremstyle{definition}
\theoremstyle{remark}

\numberwithin{equation}{section}

\title{Wave equation with Robin condition, quantitative estimates of strong unique continuation at the boundary 
}

\author { Eva Sincich \thanks{Universit\`a degli Studi di Trieste, Italy, E-mail:
\textsf{esincich@units.it}}\ \  \ \
Sergio Vessella\thanks{Universit\`a degli Studi di Firenze, Italy, E-mail:
\textsf{sergio.vessella@unifi.it}}}

\date{}

\begin{document}

\setcounter{section}{0}
\setcounter{secnumdepth}{2}

\maketitle
{\it{Dedicated to Giovanni Alessandrini on the occasion of his 60th birthday}}
\vskip 0.5 cm 
\begin{abstract}
The main result of the present paper consists in a quantitative estimate of unique continuation at the boundary for solutions to the wave equation. Such estimate is the sharp quantitative counterpart of the following strong unique continuation property: let $u$ be a solution to the wave equation that satisfies an homogeneous Robin condition on a portion $S$ of the boundary and the restriction of $u_{\mid S}$ on $S$ is flat on a segment $\{0\}\times J$ with $0\in S$  then $u_{\mid S}$ vanishes in a neighborhood of $\{0\}\times J$.
\medskip


\medskip

\end{abstract}

\section{Introduction} \label{introduction}
The strong unique continuation properties at the boundary and the related quantitative estimates have been well understood in the context of second order elliptic equations, \cite{AE}, \cite{Ku-Ny}, and in the context of second order parabolic equations \cite{EsVe}, \cite{EsFeVe}, \cite{Ve}. For instance, in the framework of elliptic equations, the doubling inequality at the boundary and three sphere inequality are the typical forms in which such quantitative estimates of unique continuation occur \cite{A-R-R-V}. Similar forms, like three cylinder inequality or two-sphere one cylinder inequality, occur in the parabolic case \cite{Ve}. In the context of hyperbolic equation, strong properties of unique continuation at the interior and the related quantitative estimates are less studied \cite{Le}, \cite{Ba-Za}, \cite{Ma}, \cite{Ve3}. Also, we recall here the papers \cite{CheDY}, \cite{CheYZ} and \cite{Ra} in which unique continuation properties are proved along and across lower dimensional manifolds for the wave equation. We refer to \cite{B-K-L1}, \cite{B-K-L1},  \cite{La-L} for recent result of quantitative estimate for hyperbolic equations. Such results are the quantitative counterpart of the unique continuation properties for equation with partially analytic coefficients proved in \cite{hormanderpaper2}, \cite{Ro-Zu} and  \cite{Ta}, see also \cite{isakovlib2}.

Quantitative estimates of strong unique continuation at the boundary are one of most important tool which enables to prove sharp stability estimates for inverse problems for PDE with unknown boundaries or with unknown boundary coefficients of Robin type, \cite{A-B-R-V}, \cite{Si2} (elliptic equations), \cite{B-Dc-Si-Ve}, \cite{CRoVe1}, \cite{DcRVe}, \cite{Ve} (parabolic equations), \cite{Ve2} (hyperbolic equations). In the context of elliptic and parabolic equations, the stability estimates that were proved are optimal \cite{Dc-R}, \cite{Al},  \cite{DcRVe}.

To the authors knowledge there exits no result in the literature concerning quantitative estimates of strong unique continuation at the boundary for hyperbolic equations.

In order to make clear what we mean, we illustrate our result in a particular and meaningful case. Let $A(x)$ be a real-valued symmetric $n\times n$, $n\geq 2$, matrix whose entries are functions of Lipschitz class satisfying a uniform ellipticity condition. Let $u$ be a solution to
\begin{equation} \label{INTR1}
\partial^2_{t}u-\mbox{div}\left(A(x)\nabla_x u\right)=0, \quad \hbox{in } B_1^+\times J,
\end{equation}
 where $B_1^+=\left\{x=(x',x_n)\in \mathbb{R}^{n}: |x|<1, x_n>0\right\}$ and $J=(-T,T)$ is an interval of $\mathbb{R}$. Assume that $u$ satisfies the following Robin condition
 \begin{equation} \label{INTR2}
A(x',0)\nabla_x u(x',0,t)\cdot \nu+\gamma(x')u(x',0,t)=0, \quad \hbox{in } B'_1\times J,
\end{equation}
where $B'_1$ is the $\mathbb{R}^{n-1}$ ball of radius $1$ centered at 0, $\nu$ denotes the outer unit normal to $B'_1$ and $\gamma$, the Robin coefficient, is of Lipschitz class. The quantitative estimate of strong unique continuation that we provide here may be briefly described as follows. Let $r\in (0,1)$ and assume that

\begin{equation}\label{Sergio1}
\sup_{t\in J}\left\Vert u(\cdot,0,t) \right\Vert_{L^2\left(B'_{r}\right)}\leq \varepsilon \quad\mbox{  and }\quad \left\Vert u(\cdot,0) \right\Vert_{H^2\left(B_1^+\right)}\leq 1,
\end{equation}
where $\varepsilon<1$.
Then

\begin{gather}
\label{SUCP-INTR}
\left\Vert u(\cdot,0,0) \right\Vert_{L^2\left(B'_{s_0}\right)} \leq C\left\vert\log \left(\varepsilon^{\theta}\right) \right\vert^{-\alpha},
\end{gather}
where $s_0\in (0,1)$, $C\geq 1$, $\alpha>0$ are constants independent of $u$ and $r$ and
\begin{equation}
\label{theta-INTR}
\theta= |\log r|^{-1}.
\end{equation}
For the precise statement of our result we refer to Theorem 2.1. Roughly speaking, in such a Theorem the half ball $B_1^+$ is replaced by the region $\{(x',x_n)\in B_1: x_n>\phi(x')\}$ where $\phi\in C^{1,1}\left(B^{\prime}_1\right)$ satisfies $\phi(0)=\left\vert\nabla_{x'}\phi(0)\right\vert=0$. In addition, $u$ satisfies the Robin condition \eqref{INTR2} on $S_1\times J$ where $S_1=\{(x',\phi(x')): x' \in B^{\prime}_1\}$.

The estimate \eqref{SUCP-INTR} is a sharp estimate from two points of view:

(i) The logarithmic character of the estimate cannot be improved as it is shown by a well-known counterexample of John for the wave equation, \cite{J};

(ii) The sharp dependence of $\theta$ by $r$. Indeed it is easy to check that the estimate \eqref{SUCP-INTR} implies that the following strong unique continuation property at the boundary holds true. Let $u$ satisfy \eqref{INTR1} and \eqref{INTR2} and assume that
  \[\sup_{t\in J}\left\Vert u(\cdot,0,t) \right\Vert_{L^2\left(B'_{r}\right)}=\mathcal{O}(r^N) \mbox{, } \forall N\in\mathbb{N} \mbox{, as  } r\rightarrow 0 \] then we have
 \[u(x',0,t)= 0 \quad \hbox{for every } (x',t)\in\mathcal{U},\]
 where $\mathcal{U}$ is a neighborhood of $\{0\}\times J$.

In order to prove the quantitative estimate \eqref{SUCP-INTR}, we have mainly refined the strategy developed in \cite{Ve3} in which the author, among various results, proved that if
\begin{equation*}
\sup_{t\in J}\left\Vert u(\cdot,t) \right\Vert_{L^2\left(B^+_{r}\right)}\leq \varepsilon \quad\mbox{  and }\quad \left\Vert u(\cdot,0) \right\Vert_{H^2\left(B_1^+\right)}\leq 1,
\end{equation*}
then
\begin{gather}
\label{SUCP-INTR2}
\left\Vert u(\cdot,0) \right\Vert_{L^2\left(B^+_{s_0}\right)} \leq C\left\vert\log \left(\varepsilon^{\theta}\right) \right\vert^{-1/6},
\end{gather}
where $\theta= |\log r|^{-1}$, $s_0\in (0,1)$, $C\geq 1$ are constants independent of $u$ and $r$ and an homogeneous Neumann boundary condition applies instead of \eqref{INTR2}.
To carry out our proof, we first adapt an argument used in \cite{Si} in the elliptic context which enable to reduce the Robin boundary condition into a Neumann boundary one. Subsequently we need a careful refinement of some arguments used in \cite{Ve3}. Actually, to fulfill our proof it is not sufficient to apply the above estimate \eqref{SUCP-INTR2}. In order to illustrate this point, a comparison with the analog elliptic context (i.e. $u$ is time independent) could be useful. In such an elliptic context \cite{Si} instead of  \eqref{Sergio1} we would have

 $$\left\Vert u(\cdot,0) \right\Vert_{L^2\left(B'_{r}\right)}\leq \varepsilon \quad\mbox{  and }\quad \left\Vert u \right\Vert_{H^2\left(B_1^+\right)}\leq 1. $$ Thus, from  stability estimates for the Cauchy problem \cite{A-R-R-V} and regularity result we would obtain the following  Holder estimate
 $$\left\Vert u \right\Vert_{L^2\left(B^+_{r}\right)}\le C\varepsilon^{\beta}$$
 where $C$ and $\beta\in (0,1)$ are independent on $u$ and $r$. By using the above estimate, the three sphere inequality at the boundary and standard regularity results we would have
  $$\left\Vert u \right\Vert_{H^1\left(B^+_{\rho}\right)}\le C\varepsilon^{\vartheta},$$
 where $0<\rho<1$ and $\vartheta \sim |\log r|^{-1}$ as $r\rightarrow 0$. Finally, by trace inequality we would obtain
 $$\left\Vert u \right\Vert_{L^2\left(B'_{\rho/2}\right)}\le C\varepsilon^{\vartheta}.$$

The application of the same argument in the hyperbolic case would lead to a \emph{loglog} type estimate instead of the desired \emph{single log} one \eqref{SUCP-INTR}. In fact, opposite to the elliptic case, in the hyperbolic context the dependence of the interior values of the solution upon the Cauchy data is \emph{logarithmic}. As a consequence, by combining such a \emph{log} dependence with the logarithmic estimate in \eqref{SUCP-INTR2} we would obtain a \emph{loglog} type estimate for $\left\Vert u(\cdot,0,0) \right\Vert_{L^2\left(B'_{s_0}\right)}$.


The plan of the paper is as follows. In Section \ref{MainRe} we state the main result of this paper. In Section \ref{SUCP_Estimates} we prove our main theorem, in Section \ref{AR} we discuss some auxiliary results and in Section \ref{conclusions} we conclude by summarizing the main steps of our proof.

\section{The main result}\label{MainRe}
\subsection{Notation and Definition} \label{sec:notation}

In several places within this manuscript it will be useful to single out one coordinate
direction. To this purpose, the following notations for
points $x\in \mathbb{R}^n$ will be adopted. For $n\geq 2$,
a point $x\in \mathbb{R}^n$ will be denoted by
$x=(x',x_n)$, where $x'\in\mathbb{R}^{n-1}$ and $x_n\in\mathbb{R}$.
Moreover, given  $r>0$, we will denote by $B_r$, $B'_r$ $\widetilde{B}_r$ the ball of $\mathbb{R}^{n}$, $\mathbb{R}^{n-1}$ and $\mathbb{R}^{n+1}$ of radius $r$ centred at 0.
For any open set $\Omega\subset\mathbb{R}^n$ and any function (smooth enough) $u$  we denote by $\nabla_x u=(\partial_{x_1}u,\cdots, \partial_{x_n}u)$ the gradient of $u$. Also, for the gradient of $u$ we use the notation $D_xu$. If $j=0,1,2$ we denote by $D^j_x u$ the set of the derivatives of $u$ of order $j$, so $D^0_x u=u$, $D^1_x u=\nabla_x u$ and $D^2_xu$ is the Hessian matrix $\{\partial_{x_ix_j}u\}_{i,j=1}^n$. Similar notation are used whenever other variables occur and $\Omega$ is an open subset of $\mathbb{R}^{n-1}$ or a subset of $\mathbb{R}^{n+1}$. By $H^{\ell}(\Omega)$, $\ell=0,1,2$ we denote the usual Sobolev spaces of order $\ell$, in particular we have $H^0(\Omega)=L^2(\Omega)$.

For any interval $J\subset \mathbb{R}$ and $\Omega$ as above we denote by
 \[\mathcal{W}\left(J;\Omega\right)=\left\{u\in C^0\left(J;H^2\left(\Omega\right)\right): \partial_t^\ell u\in C^0\left(J;H^{2-\ell}\left(\Omega\right)\right), \ell=1,2\right\}.\]

We shall use the letters $C,C_0,C_1,\cdots$ to denote constants. The value of the constants may change from line to line, but we shall specified their dependence everywhere they appear.

\subsection{Statements of the main results}\label{QEsucp}
Let $A(x)=\left\{a^{ij}(x)\right\}^n_{i,j=1}$ be a real-valued symmetric $n\times n$ matrix whose entries are measurable functions and they satisfy the following conditions for given constants $\rho_0>0$, $\lambda\in(0,1]$ and $\Lambda>0$,
\begin{subequations}
\label{1-65}
\begin{equation}
\label{1-65a}
\lambda\left\vert\xi\right\vert^2\leq A(x)\xi\cdot\xi\leq\lambda^{-1}\left\vert\xi\right\vert^2, \quad \hbox{for every } x, \xi\in\mathbb{R}^n,
\end{equation}
\begin{equation}
\label{2-65}
\left\vert A(x)-A(y)\right\vert\leq\frac{\Lambda}{\rho_0} \left\vert x-y \right\vert, \quad \hbox{for every } x, y\in\mathbb{R}^n.
\end{equation}
\end{subequations}

Let $\phi$ be a function belonging to $C^{1,1}\left(B^{\prime}_{\rho_0}\right)$ that satisfies

\begin{subequations}
\label{phi}
\begin{equation}
\label{phi_0}
\phi(0)=\left\vert\nabla_{x'}\phi(0)\right\vert=0
\end{equation}
\begin{equation}
\label{phi_M0}
\left\Vert\phi\right\Vert_{C^{1,1}\left(B^{\prime}_{\rho_0}\right)}\leq E\rho_0,
\end{equation}
\end{subequations}
where

\begin{equation*}
\left\Vert\phi\right\Vert_{C^{1,1}\left(B^{\prime}_{\rho_0}\right)}=\left\Vert\phi\right\Vert_{L^{\infty}\left(B^{\prime}_{\rho_0}\right)}
+\rho_0\left\Vert\nabla_{x'}\phi\right\Vert_{L^{\infty}\left(B^{\prime}_{\rho_0}\right)}+
\rho_0^2\left\Vert D_{x'}^2\phi\right\Vert_{L^{\infty}\left(B^{\prime}_{\rho_0}\right)}.
\end{equation*}
For any $r\in (0,\rho_0]$ denote by
\[
K_{r}:=\{(x',x_n)\in B_{r}: x_n>\phi(x')\}
\]
and
\[S_{r}:=\{(x',\phi(x')): x' \in B^{\prime}_{r}\}.\]

We assume that the Robin coefficient $\gamma$ belongs to $C^{0,1}(S_{\rho_0})$ and for a given $\bar{\gamma}>0$ is such that

\begin{equation}
\label{gamma_0}
\left\Vert\gamma\right\Vert_{C^{0,1}\left(S_{\rho_0}\right)}\leq \bar{\gamma}\ .
\end{equation}

Let $U\in \mathcal{W}\left([-\lambda\rho_0,\lambda\rho_0];K_{\rho_0}\right)$ be a solution to

\begin{equation} \label{4i-65Boundary}
\partial^2_{t}U-\mbox{div}\left(A(x)\nabla_x U\right)=0, \quad \hbox{in } K_{\rho_0}\times(-\lambda\rho_0,\lambda\rho_0),
\end{equation}
satisfying the following Robin condition
\begin{equation} \label{RobinBoundary}
A\nabla_x U\cdot \nu + \gamma U=0, \quad \hbox{on } S_{\rho_0}\times(-\lambda\rho_0,\lambda\rho_0),
\end{equation}
where $\nu$ denotes the outer unit normal to $S_{\rho_0}$.

Let $r_0\in (0,\rho_0]$ and denote by

\begin{equation} \label{4ii-65Boundary}
\varepsilon=\sup_{t\in (-\lambda\rho_0,\lambda\rho_0)}\left(\rho_0^{-n+1}\int_{S_{r_0}}U^2(\sigma,t)d\sigma\right)^{1/2}
\end{equation}
and

\begin{equation} \label{4iii-65Boundary}
H=\left(\sum_{j=0}^2\rho_0^{j-n}\int_{K_{\rho_0}}\left\vert D_x^jU(x,0)\right\vert^2 dx\right)^{1/2}.
\end{equation}

\begin{theo}
\label{5-115Boundary}
Let \eqref{1-65} be satisfied. Let $U\in\mathcal{W}\left([-\lambda\rho_0,\lambda\rho_0];K_{\rho_0}\right)$ be a solution to \eqref{4i-65Boundary} satisfying \eqref{4ii-65Boundary} and \eqref{4iii-65Boundary}. Assume that $u$ satisfies \eqref{RobinBoundary}. There exist constants $\overline{s}_0\in (0,1)$ and $C\geq 1$ depending on $\lambda$, $\Lambda$ and $E$ only such that for every $0<r_0\leq \rho\leq \overline{s}_0 \rho_0$ the following inequality holds true

\begin{gather}
\label{SUCPBoundary}
\left\Vert U(\cdot,0) \right\Vert_{L^2\left(S_{\rho}\right)}\leq \frac{C\left(\rho_0\rho^{-1}\right)^{C}(H+e\varepsilon)}{\left(\widetilde{\theta}\log \left( \frac{H+e\varepsilon}{\varepsilon}\right)\right)^{1/6}} ,
\end{gather}
where
\begin{equation}
\label{theta-1new}
\widetilde{\theta}=\frac{\log (\rho_0/C\rho)}{\log (\rho_0/r_0)}.
\end{equation}
\end{theo}

From now on we shall refer to the a priori data as the following set of quantities: $\lambda, \Lambda, \rho_0, E, \bar{\gamma}$.

\section{Proof of Theorem \ref{5-115Boundary}} \label{SUCP_Estimates}

In what follows we use the following
\begin{prop}\label{solpos}
There exists a radius $r_1>0$ depending on the a priori data only, such that the problem
\begin{equation}\label{positivesolution}
\left\{
\begin{array}
{lcl}
\mbox{div}(A \nabla \psi)=0\ ,& \mbox{in $K_{r_1}$ ,}
\\
A\nabla \psi\cdot\nu + \gamma\psi=0 \ ,& \mbox{in $S_{r_1}$ ,}
\end{array}
\right.
\end{equation}
admits a solution $\psi \in H^1(K_{r_1})$ satisfying
\begin{eqnarray}\label{LB}
\psi(x)\ge 1 \ \ \mbox{for every}\ \ x\in K_{r_1}.
\end{eqnarray}
Moreover, there exists a constant $\bar{\psi}>0$ depending on the a priori data only, such that
\begin{eqnarray}\label{UB}
\| \psi\|_{C^1(K_{r_1})}\le \bar{\psi}\ .
\end{eqnarray}
\end{prop}
\begin{proof} See Section \ref{AR}
\end{proof}
Let $r_1$ and $\psi$ be the radius and the function introduced in Proposition \ref{solpos}.
Denoting with
\begin{eqnarray}
u^\star=\frac{U}{\psi}
\end{eqnarray}
it follows that $u^\star\in \mathcal{W}\left([-\lambda r_1,\lambda r_1];K_{r_1}\right)$ is a solution to

\begin{equation} \label{4i-65BoundaryN}
\psi^2(x)\partial^2_{t}u^\star-\mbox{div}\left(A^\star(x)\nabla_x u^\star\right)=0, \quad \hbox{in } K_{r_1}\times(-\lambda r_1,\lambda r_1),
\end{equation}
satisfying the following Neumann condition
\begin{equation} \label{NeumannBoundary}
A^\star\nabla_x u^\star\cdot \nu =0, \quad \hbox{on } S_{r_1}\times(-\lambda r_1,\lambda r_1),
\end{equation}
where $\nu$ denotes the outer unit normal to $S_{r_1}$ and $A^\star(x)=\psi^2(x)A(x)$.
Repeating the arguments in \cite[Subsection 3.2]{Ve3} (partly based on the techniques introduced in \cite{AE}), we can assume with no loss of generality that $A^\star(0)=I$ with $I$ identity matrix $n\times n$ and we  infer that there exist $\rho_1, \rho_2$ and a function $\phi\in C^{1,1}(\overline{B}_{\rho_2}, \mathbb{R}^n)$ such that

\begin{subequations}
\label{Phi}
\begin{equation}
\label{Phia}
\Phi(B_{\rho_2})\subset B_{\rho_1}
\end{equation}
\begin{equation}
\label{Phib}
\Phi(y,0)=(y',\phi(y'))\ .
\end{equation}
\begin{equation}
\label{Phic}
C^{-1}\leq |\textrm{det}D\Phi(y)|\leq C,\quad\hbox{ for every }  y\in B_{\rho_2}.
\end{equation}

\end{subequations}

Let us define the matrix $\overline{A}(y)=\{\overline{a}(y)\}_{i,j=1}^n$ as follows (below $(D\Phi^{-1})^{tr}$ denotes the transposed matrix of  $(D\Phi^{-1})$)

\begin{equation*}
\overline{A}(y)=|\textrm{det}D\Phi(y)|(D\Phi^{-1})(\Phi(y)) A^\star(\Phi(y))(D\Phi^{-1})^{tr}(\Phi(y)),
\end{equation*}
\begin{equation}
z(y,t)=u^\star(\Phi(y),t)\
\end{equation}

\begin{equation}\label{defu}
u(y,t)=z(y',|y_n|,t)\
\end{equation}
and hence we get that $u$ is a solution to
\begin{equation} \label{4i-65Ball}
q(y)\partial^2_{t}u-\mbox{div}\left(\tilde {A}(y)\nabla{u}\right)=0, \quad \hbox{in } B_{\rho_2}\times(-\lambda \rho_2,\lambda \rho_2),
\end{equation}
where for every $y\in B_{\rho_2}$ we denote
\begin{equation*}
{q}(y)=|\textrm{det}D\Phi(y',|y_n|)| \psi^2(y',|y_n|),
\end{equation*}
and $\tilde A(y)=\{\tilde a_{ij}(y)\}_{i,j=1}^n$ is the matrix whose entries are given by
\begin{subequations}
\label{Atilde}
\begin{equation}
\label{Atildeprimoblocco}
\tilde a_{ij}(y',y_n)=\overline{a}_{ij}(y',|y_n|),\quad\hbox{ if either}  i,j\in\{1,\ldots,n-1\}\hbox{, or }i=j=n,
\end{equation}
\begin{equation}
\label{Atildesecondoblocco}
\tilde a_{nj}(y',y_n)=\tilde a_{jn}(y',y_n)=\textrm{sgn}(y_n)\overline{a}^{nj}(y',|y_n|),\quad\hbox{ if }  1\leq j\leq n-1.
\end{equation}
\end{subequations}

From \eqref{1-65a}, \eqref{2-65}, \eqref{Phic}, \eqref{LB} and \eqref{UB} there exist constants $\tilde{\Lambda},\tilde{\lambda}>0$ depending on the a priori data only such that

\begin{subequations}
\label{4-65}
\begin{equation}
\label{4-65a}
\tilde{\lambda}\left\vert\xi\right\vert^2\leq \tilde{A}(y)\xi\cdot\xi\leq\tilde{\lambda}^{-1}\left\vert\xi\right\vert^2, \quad \hbox{for every } y\in B_{\rho_2}, \xi\in\mathbb{R}^n,
\end{equation}
\begin{equation}
\label{4-65b}
\left\vert\tilde{A}(y_1)-\tilde{A}(y_2)\right\vert\leq\frac{\tilde{\Lambda}}{\rho_0} \left\vert y_1-y_2 \right\vert, \quad \hbox{for every } y_1,y_2\in B_{\rho_2}
\end{equation}
\end{subequations}
and

\begin{subequations}
 \label{3-65a}
 \begin{equation} \label{3-65aa}
\tilde{\lambda}\leq q(y)\leq\tilde{\lambda}^{-1}, \quad \hbox{for every } y\in B_{\rho_2}\ ,
\end{equation}
\begin{equation}
\label{3-65ab}
\left\vert q(y_1)-q(y_2)\right\vert\leq\frac{\tilde{\Lambda}}{\rho_0} \left\vert y_1-y_2 \right\vert, \quad \hbox{for every } y_1,y_2\in B_{\rho_2}.
\end{equation}
\end{subequations}

Let us recall that, by construction, the function $u$ in \eqref{defu} is even w.r.t. the variable $y_n$ and moreover with no loss of generality we may assume that $u$ (up to replacing it with its even part w.r.t the variable $t$ as in \cite{Ve3}) is even w.r.t. $t$ also. From now for the sake of simplicity we shall assume that $\rho_2=1$.

By \eqref{4ii-65Boundary} and by \eqref{4iii-65Boundary} we have that there exist $C_1,C_2>0$ constants depending on the a priori data only such that
\begin{equation} \label{4iv-65Boundary}
\epsilon=\sup_{t\in (-\lambda,\lambda)}\left(\int_{B'_{r_0}}u^2(y',0,t)dy'\right)^{1/2}\le C_1\varepsilon
\end{equation}
\begin{equation} \label{4v-65Boundary}
H_1=\left(\sum_{j=0}^2\int_{B_1}\left\vert D_x^ju(y,0)\right\vert^2 dy\right)^{1/2}\le C_2 H
\end{equation}

As in \cite{Ve3}, let $\widetilde{u}_0$ be an even extension w.r.t. $y_n$ of the function $u_0:=u(\cdot,0)$ such that $\widetilde{u}_0\in H^2\left(B_2\right)\cap H_0^1\left(B_2\right)$ and

\begin{equation} \label{2-70}
\|\widetilde{u}_0\|_{H^2\left(B_2\right)}\leq CH_1,
\end{equation}
where $C$ is an absolute constant.

Let us denote by $\lambda_j$, with $0<\lambda_1\leq\lambda_2\leq \cdots\leq\lambda_j\leq\cdots$ the eigenvalues associated to the Dirichlet problem

\begin{equation}
\label{1-71}
\left\{\begin{array}{ll}
\mbox{div}\left(\tilde{A}(y)\nabla_y v\right)+\omega q(y)v=0, & \textrm{in }B_2,\\[2mm]
v\in H_0^1\left(B_2\right)\ .
\end{array}\right.
\end{equation}
and by $e_j(\cdot)$ the corresponding eigenfunctions normalized by

\begin{equation} \label{2-71}
\int_{B_2}e^2_j(y)q(y)dy=1.
\end{equation}

Let us stress that we may choose the eigenfunctions $e_j$ to be even w.r.t $y_n$ (see Remark \ref{autofpari} in Section \ref{AR}).
By \eqref{1-65a}, \eqref{3-65a} and Poincar\'{e} inequality we have for every $j\in\mathbb{N}$

\begin{gather}
 \label{2-71NEW}
\lambda_j=\int_{B_2} \tilde{A}(y)\nabla_x e_j(y)\cdot \nabla_y e_j(y) dy\geq c\lambda^2 \int_{B_2}e^2_j(y)q(y)dy=c\lambda^2
\end{gather}
where $c$ is an absolute constant.
Denote by

\begin{equation} \label{4-71}
\alpha_j :=\int_{B_2}\widetilde{u}_0(y) e_j(y)q(y)dy,
\end{equation}
and let

\begin{equation} \label{3-71}
\widetilde{u}(y,t):=\sum_{j=1}^{\infty}\alpha_j e_j(y)\cos\sqrt{\lambda_j} t.
\end{equation}
By Proposition 3.3 in \cite{Ve3} we have that
\begin{eqnarray}\label{bound}
\sum_{j=1}^{\infty}(1+\lambda_j)^2 \alpha^2_j\le C H_1^2 \ ,
\end{eqnarray}
where $C>0$ depends on $\tilde{\lambda}$ and $\tilde{\Lambda}$ only.

Moreover, as a consequence of the uniqueness for the Cauchy problem for the equation \eqref{4i-65Ball} (see $(3.9)$ in \cite{Ve3} for a detailed discussion) we have that
\begin{equation}\label{CP}
\tilde{u}(y,t)=u(y,t)\ \ \ \mbox{for}\ \ |y|+\tilde{\lambda}^{-1}|t|<1 \  .
\end{equation}

 We define for any $\mu \in (0,1]$ and for any $k\in \mathbb{N}$ the following mollified form of the Boman transformation of $\widetilde{u}(y,\cdot)$ \cite{Bo}
\begin{equation}
\label{2-76}
\widetilde{u}_{\mu,k}(x)=\int_{\mathbb{R}}\widetilde{u}(x,t)\varphi_{\mu,k}(t)dt \mbox{, for } x\in B_2.
\end{equation}
where $\{\varphi_{\mu,k}\}_{k=1}^{\infty}$ is a suitable sequence of mollifiers, \cite[Section 3.1]{Ve3}, such that $\mbox{supp }\varphi_{\mu,k}\subset\left[-\frac{\lambda(\mu+1)}{4},\frac{\lambda(\mu+1)}{4}\right]$, $\varphi_{\mu,k}\geq 0$,  $\varphi_{\mu,k}$ even function and such that $\int_{\mathbb{R}}\varphi_{\mu,k}(t)dt=1$.

From now on we fix $\overline{\mu}:=k^{-\frac{1}{6}}$ for $k\geq 1$ and we denote

\begin{equation}
\label{u-k}
\widetilde{u}_k:=\widetilde{u}_{\overline{\mu},k}.
\end{equation}

By Proposition 3.3 im \cite{Ve3}, it follows that

\begin{equation}
\label{2'-97}
\left\Vert u(\cdot,0)-\widetilde{u}_{\mu,k} \right\Vert_{L^2 \left(B_{1}\right)}\leq C H k^{-1/6},
\end{equation}

Let
\begin{equation*}
\label{2-74}
\widehat{\varphi}_{\overline{\mu},k}(\tau)=\int_{\mathbb{R}}\varphi_{\overline{\mu},k}(t)e^{-i\tau t}dt=\int_{\mathbb{R}}\varphi_{\overline{\mu},k}(t) \cos \tau t dt \mbox{, } \tau\in\mathbb{R}.
\end{equation*}
Let us introduce now, for every $k\in \mathbb{N}$ an even function $g_k\in C^{1,1}(\mathbb{R})$ such that if $|z|\leq k$ then we have $g_k(z)=\cosh z$, if  $|z|\geq 2k$ then we have $g_k(z)=\cosh 2k$ and such that it satisfies the condition

\begin{equation}
\label{2-80}
\left\vert g_k(z) \right\vert+\left\vert g^{\prime}_k(z) \right\vert+\left\vert g^{\prime\prime}_k(z) \right\vert\leq ce^{2k} \mbox{, for every } z\in\mathbb{R},
\end{equation}
where $c$ is an absolute constant.

\bigskip
Let us introduce the following quantities
\begin{subequations}
\label{88}
\begin{equation}
\label{7-81}
h_k(z)=e^{2k}\min\left\{1,\left(4\pi\lambda^{-1}|z|\right)^{2k}\right\}\ \ , \ \ z\in\mathbb{R}
\end{equation}
\begin{equation}
\label{5-81}
f_{k}(y,z)=\sum_{j=1}^{\infty}\lambda_j\alpha_j \widehat{\varphi}_{\overline{\mu},k}\left(\sqrt{\lambda_j}\right)\left(g^{\prime\prime}_k\left(z\sqrt{\lambda_j}\right)-
g_k\left(z\sqrt{\lambda_j}\right)\right)e_j(y)\ ,\   y\in B_2\ ,z\in \mathbb{R},
\end{equation}
\begin{equation}
\label{6-81}
F_{k}(y,t,z)=\sum_{j=1}^{\infty}\alpha_j\sqrt{\lambda_j}\gamma_k(z\sqrt{\lambda_j})\sin(\sqrt{\lambda_j}t)e_j(y)\ , \ y\in B_2\ , t,z\in \mathbb{R},
\end{equation}
\begin{equation}
\label{8-81}
\gamma_k(z\sqrt{\lambda_j})= g^{\prime\prime}_k(z\sqrt{\lambda_j})- g_k(z\sqrt{\lambda_j}), \ \ z\in\mathbb{R}.
\end{equation}
\end{subequations}

\begin{prop}\label{2-81prop}
Let
\begin{equation}
\label{3-80}
v_{k}(y,z):=\sum_{j=1}^{\infty}\alpha_j \widehat{\varphi}_{\overline{\mu},k}\left(\sqrt{\lambda_j}\right)g_k\left(y\sqrt{\lambda_j}\right)e_j(z) \mbox{ , for  } (y,z)\in B_2\times\mathbb{R}.
\end{equation}
We have that $v_{k}(\cdot,z)$ belongs to $H^2\left(B_2\right)\cap H_0^1\left(B_2\right)$ for every $y\in \mathbb{R}$, $v_{k}(y,z)$ is an even function with respect to $z$ and it satisfies
\begin{equation}
\label{4-5-6-81}
\left\{\begin{array}{ll}
q(y)\partial^2_{z}v_{k}+\mbox{div}\left(\tilde{A}(y)\nabla_x v_{k}\right)=f_{k}(y,z), \quad \hbox{in } B_2\times \mathbb{R},\\[2mm]
v_{k}(\cdot,0)=\widetilde{u}_{k},\quad \hbox{in } B_2.
\end{array}\right.
\end{equation}

Moreover we have

\begin{equation}
\label{3-81}
\sum_{j=0}^{2}\|\partial^{j}_yv_{k}(\cdot,z)\|_{H^{2-j}\left(B_2\right)}\leq CH e^{2k} \mbox{, for every  } z\in \mathbb{R},
\end{equation}

\begin{equation}
\label{2-81}
\|f_{k}(\cdot,z)\|_{L^2\left(B_2\right)}\leq CH e^{2k}\min\left\{1,\left(4\pi\lambda^{-1}|z|\right)^{2k}\right\} \mbox{, for every  } z\in \mathbb{R},
\end{equation}

\begin{equation}
\label{4-81}
\|F_{k}(\cdot,0,t,z)\|_{H^{\frac{1}{2}}\left(B^{\prime}_1\right)}\leq C H_1 h_k(z)\mbox{, for every  } t,z\in \mathbb{R},
\end{equation}

 where
 $C$ depends on $\tilde{\lambda}$ and $\Lambda$ only.

\end{prop}
\begin{proof}
Except for the inequality \eqref{4-81} which is discussed below, the proofs of the remaining results follow along the lines of  Proposition 3.4 in \cite{Ve3}.
From the arguments in Proposition 3.4 in \cite{Ve3} we deduce that
\begin{eqnarray}
|\gamma_k(z\sqrt{\lambda_j})|\le c h_k(z)
\end{eqnarray}
where $c>0$ is an absolute constant constant, which in turn implies that
\begin{eqnarray}\label{l2}
\|F_{k}(\cdot,0,t,z)\|_{L^{2}\left(B_2\right)}\leq c h_k^2 \sum_{j=1}^{\infty}\alpha_j^2\lambda_j\le C {H^2_1}{h^2_k}(z)
\end{eqnarray}
with $C>0$ constant depending on $\tilde{\lambda}$.

From \eqref{4-65a} we have
\begin{eqnarray}\label{seminorma}
&&\tilde{\lambda}\int_{B_2}|\nabla_y F_k(y,t,z)|^2 \mbox{d}y\le \int_{B_2} \tilde{A}(y)\nabla_y F_k(y,t,z)\cdot \nabla_y F_k(y,t,z)\mbox{d}y =\\ \nonumber
&&=\sum_{j=1}^{\infty}\alpha_j \sqrt{\lambda_j}\sin(\sqrt{\lambda_j}t)\gamma_k(z\sqrt{\lambda_j})\int_{B_2}\tilde{A}(y)\nabla_y e_j(y)\cdot \nabla_y F_k(y,t,z)\mbox{d}y=\\ \nonumber
&&=\sum_{j=1}^{\infty}\alpha_j \sqrt{\lambda_j}\sin(\sqrt{\lambda_j}t)\gamma_k(z\sqrt{\lambda_j})\int_{B_2}\lambda_j q(y) e_j(y)F_k(y,t,z)\mbox{d}y=\\ \nonumber
&&=\sum_{j=1}^{\infty}\alpha^2_j {\lambda^2_j}(\sin(\sqrt{\lambda_j}t)\gamma_k(z\sqrt{\lambda_j}))^2\le \sum_{j=1}^{\infty}\alpha^2_j {\lambda^2_j}(ch_k(z))^2\le C H^2_1 h^2_k(z)\
\end{eqnarray}
where $C>0$ is a constant depending on $\tilde{\lambda}$ and $\tilde{\Lambda}$ only.

Combining \eqref{l2} and \eqref{seminorma} we get
\begin{eqnarray}
\|F_{k}(\cdot,t,z)\|_{H^{1}\left(B_2\right)}\leq C H_1 h_k(z)\  ,
\end{eqnarray}
which in view of standard trace estimates leads to
\begin{eqnarray}
\|F_{k}(\cdot,0,t,z)\|_{H^{\frac{1}{2}}\left(B^{\prime}_1\right)}\leq C H_1 h_k(z)\  .
\end{eqnarray}

\end{proof}

Let us now consider a function $\Phi \in L^2(B^{\prime}_{r_0})$ and let us define for any $(t,z)\in R=\{(t,z)\in \mathbb{R}^2:|t|<\tilde{\lambda}, |z|<1 \}$
\begin{eqnarray}\label{wk}
w_k(t,z)=\int_{B^{\prime}_{r_0}} W_k(y',0,t,z)\Phi(y')\mbox{d}y'
\end{eqnarray}
where
\begin{eqnarray}\label{Wk}
W_k(y,t,z)=\sum_{j=1}^{\infty}\alpha_j \cos(\sqrt{\lambda_j}t)g_k(z\sqrt{\lambda_j})e_j(y)\ .
\end{eqnarray}
Note that from \eqref{3-80} we have
\begin{eqnarray}\label{v-k}
v_k(y,z)=\int_{\mathbb{R}}{\varphi}_{\bar{\mu},k}(t)W_k(y,t,z)\mbox{d}t\ \ ,
\end{eqnarray}

\begin{prop}\label{3-81prop}
We have that $w_{k}(\cdot,\cdot)$ belongs to $H^1\left(R\right)$ is a weak solution to
\begin{equation}\label{0-89}
\Delta_{t,z}w_k(t,z)=-\partial_t \tilde{F_k}(t,z)
\end{equation}
satisfying
\begin{subequations}
\label{89}
\begin{equation}
\label{1-89}
|w_k(t,0)|\le \epsilon \|\Phi\|_{L^2(B^{\prime}_{r_0})}
\end{equation}
\begin{equation}
\label{2-89}
\partial_{z}w_k(t,0)=0
\end{equation}

\end{subequations}
where
\begin{eqnarray}
\tilde{F_k}(t,z)=\int_{B^{\prime}_{r_0}}F_k(y',0,t,z)\Phi(y')\mbox{d}y' \ .
\end{eqnarray}
Moreover, for any $(t,z)\in R$ we have that
\begin{subequations}
\label{90}
\begin{equation}
\label{1-90}
|w_k(t,z)|\le C H_1 e^{2k} \|\Phi\|_{L^2(B^{\prime}_{r_0})}
\end{equation}
\begin{equation}
\label{2-90}
|\tilde{F_k}(t,z)|\le C H_1 h_k(z)\|\Phi\|_{L^2(B^{\prime}_{r_0})}
\end{equation}
\end{subequations}
where $C>0$ is a constant depending on $\tilde{\lambda}$ and $\tilde{\Lambda}$ only.

\end{prop}

\begin{proof}
We start by proving \eqref{0-89}. To this aim we consider a test function  $\phi\in H^1_0(R)$ and by integration by parts we get
\begin{eqnarray}
&&\int_R \nabla_{t,z}w_k \cdot \nabla \phi \mbox{d}t\mbox{d}y= \\ \nonumber
&&\sum_{j=1}^{\infty}\int_{R}\lambda_j \alpha_j <e_j,\Phi>(g_k(z\sqrt{\lambda_j}) - g''_k(z\sqrt{\lambda_j}))\cos(\sqrt{\lambda_j}t)\phi(t,z)\mbox{d}t\ \mbox{d}z \\ \nonumber
&& \sum_{j=1}^{\infty}-\int_{R}\partial_{t}\left(\sqrt{\lambda_j} \alpha_j <e_j,\Phi>\gamma_k(z\sqrt{\lambda_j})\sin(\sqrt{\lambda_j}t)\right)\phi(t,z)\ \mbox{d}t\mbox{d}z
\end{eqnarray}
where we mean  $<e_j,\Phi>= \int_{B^{\prime}_{r_0}}e_j(y',0)\Phi(y'))\mbox{d}y'$\ .
Again by integration by parts with respect to the variable $t$ we get

\begin{eqnarray}
&&\int_R \nabla_{t,z}w_k \cdot \nabla \phi \mbox{d}t\mbox{d}y= \int_R \left (\int_{B^{\prime}_{r_0}}F_k(y',0,t,z)\Phi(y')\mbox{d}y'\right)\partial_{t}\phi \ \mbox{d}t\mbox{d}z
\end{eqnarray}
and hence \eqref{0-89} follows.

Let us now prove \eqref{1-89} and \eqref{2-89}.
We have that by \eqref{3-71}
\begin{eqnarray}
w_k(t,0)=\int_{B^{\prime}_{r_0}}\tilde{u}(y',0,t)\varphi(y')\mbox{d}y'\ .
\end{eqnarray}
Hence by \eqref{CP} and \eqref{4iv-65Boundary} we have that
\begin{gather}
|w_k(t,0)|\le \left(\int_{B^{\prime}_{r_0}}|\tilde{u}(y',0,t)|^2 \mbox{d}y'\right)^{\frac{1}{2}}\|\Phi\|_{L^2({B^{\prime}_{r_0})}}\le \epsilon \|\Phi\|_{L^2({B^{\prime}_{r_0})}}
\end{gather}
By \eqref{Wk} we also get that
\begin{eqnarray}
\partial_{z}w_k(t,0)=\int_{B^{\prime}_{r_0}}W_k(y',0,t,z)|_{z=0}\Phi(y')\mbox{d}y'=0\ .
\end{eqnarray}
Let us now prove \eqref{1-90}. By a standard trace inequality, by \eqref{bound} and by \eqref{2-80} we have
\begin{gather}
|w_k(t,z)|\le
\|W_k\|_{H^1(B_2)}\|\Phi\|_{L^2(B^{\prime}{r_0})}\le\\ \nonumber
\le C e^{2k}\left(\sum_{j=1}^{\infty}(1+\lambda_j)\alpha^2_j \right)^{\frac{1}{2}}\|\Phi\|_{L^2(B^{\prime}{r_0})}
\le C H_1 e^{2k}\|\Phi\|_{L^2(B^{\prime}{r_0})}\ .
\end{gather}
Finally \eqref{2-90} follows from \eqref{4-81}.

\end{proof}

\begin{prop}\label{CauchyProblem}
Let $w_k$ be the function introduced in \eqref{wk}, then we have that
\begin{eqnarray}\label{stimaCauchy}
|w_k(t,z)|\le C r_0^{\frac{1}{2}}\sigma_k \|\Phi\|_{L^2(B'_{r_0})}\ \ \mbox{for any } |t|\le \frac{\tilde{\lambda}}{2}, |z|\le \frac{r_0}{8}
\end{eqnarray}
where
\begin{eqnarray}\label{sigmak}
\sigma_k=\left(\epsilon + H_1 (Cr_0)^{2k} \right)^{\beta}\left(H_1 (Cr_0)^{2k} + H_1e^{2k}  \right)^{1-\beta} \ .
\end{eqnarray}

\end{prop}

\begin{proof}
We notice that by \eqref{0-89} and by a standard local boundedness estimate it follows that for any $t_0\in(-\frac{\tilde{\lambda}}{2}, \frac{\tilde{\lambda}}{2})$ we have
\begin{eqnarray}
\|w_k\|_{L^{\infty}(B^{(2)}_{\frac{r_0}{8}}(t_0,0))}\le \frac{1}{r_0}\|w_k\|_{L^{2}(B^{(2)}_{\frac{r_0}{4}}(t_0,0))}
\end{eqnarray}
where we denote $B^{(2)}_{r}(t_0,0)=\{(t,z)\in \mathbb{R}^2: |t-t_0|^2+ |z|^2\le r^2 \}$ for any $r>0$.

Let $\tilde{w}_k\in H^1(B_{\frac{r_0}{8}}^{(2)}(t_0,0))$ be the solution to the following Dirichlet problem

\begin{equation}
\label{dirichlet}
\left\{\begin{array}{ll}
\Delta_{t,z} {\tilde{w}_k}= -\partial_t \tilde{F}_k(t,z)\quad \hbox{in } B_{\frac{r_0}{8}}^{(2)}(t_0,0),\\[2mm]
\tilde{w_k}=0 \quad \hbox{on } \partial B_{\frac{r_0}{8}}^{(2)}(t_0,0)\ .
\end{array}\right.
\end{equation}

We observe that being $\partial_t \tilde{F}_k(t,z)$ odd with respect the variable $z$, we have that $\tilde{w}_k$ is odd with respect the variable $z$ as well. Moreover, we have that
$\partial_{z}{\tilde{w}_k}(t,z)=0$ on $B_{\frac{r_0}{8}}^{(1)}$
where we denote $B_{r}^{(1)}=(t_0-r, t_0+r)\times \{0\}$ for any $r>0$.

Now denoting
\begin{eqnarray}
\hat{w_k}=w_k - \tilde{w}_k
\end{eqnarray}
we have that
\begin{equation}
\label{35}
\left\{\begin{array}{ll}
\Delta_{t,z} {\hat{w}_k}= 0\quad \hbox{in } B_{\frac{r_0}{8}}^{(2)}(t_0,0),\\[2mm]
\hat{w_k}=0 \quad \hbox{on }\ B_{\frac{r_0}{8}}^{(1)}.
\end{array}\right.
\end{equation}
By the argument in Proposition 3.5 of \cite{Ve3},  which in turn are based on well-known stability estimates for the Cauchy problem (see for instance \cite{A-R-R-V}), it follows that
\begin{gather}\label{daCP}
\int_{B_{\frac{r_0}{32}}^{(2)}(t_0,0)}|\hat{w_k}|^2\le  C \left( \int_{B_{\frac{r_0}{8}}^{(2)}(t_0,0)} |\hat{w_k}|^2 \right)^{1-\beta} \left( \int_{B_{\frac{r_0}{16}}^{(1)}(t_0,0)} |\hat{w_k}|^2 \right)^{\beta}\ .
\end{gather}
Furthermore we have that by \eqref{1-89}, \eqref{2-90} and \eqref{1-90}

\begin{subequations}
\label{92}
\begin{equation}
\label{1-92}
\|\hat{w_k}\|_{L^2(B_{\frac{r_0}{16}}^{(1)}(t_0,0))}\le C (\epsilon + H_1(Cr_0)^{2k}) \|\Phi\|_{L^2(B^{\prime}_{r_0})}
\end{equation}
\begin{equation}
\label{2-92}
\|\hat{w_k}\|_{L^2(B_{\frac{r_0}{8}}^{(2)}(t_0,0))}\le C \left(  H_1 e^{2k}+H_1(Cr_0)^{2k}\right)\|\Phi\|_{L^2(B^{\prime}_{r_0})}
\end{equation}
\end{subequations}
where $C>0$ is a constant depending on the a priori data only.
Inserting \eqref{1-92} and \eqref{2-92} in \eqref{daCP} we get the thesis.

\end{proof}

\begin{prop}
Let $v_k$ be defined in \eqref{3-80}, then we have
\begin{eqnarray}\label{1-93}
\int_{B'_{r_0}}|v_k(y',0,z)|^2\mbox{d}y'\le (C r_0^{-\frac{1}{2}}\sigma_k)^2\ ,
\end{eqnarray}
where $C>0$ depends on $\tilde{\lambda}$ and $\tilde{\Lambda}$ only.

\end{prop}

\begin{proof}
From \eqref{wk}, \eqref{stimaCauchy} and the dual characterization of the norm, we have that
\begin{eqnarray}\label{dc}
\int_{B'_{r_0}}|W_k(y',0,t,z)|^2\mbox{d}y'\le (C r_0^{-\frac{1}{2}}\sigma_k)^2\ ,
\end{eqnarray}
for $|t|\le \frac{\tilde{\lambda}}{2}, |z|\le \frac{r_0}{8}$\ .
On the other hand by using equality \eqref{v-k}, we have that
\begin{eqnarray}
&&|v_k(y',0,z)|^2\le \left|\int_{\frac{-\tilde{\lambda}(\bar{\mu}+1)}{4}}^{\frac{\tilde{\lambda}(\bar{\mu}+1)}{4}} {\varphi}_{\bar{\mu},k}(t)W_k(y',0,t,z)\mbox{d}t\right|^2\le \\ \nonumber
&&\left(\int_{\frac{-\tilde{\lambda}(\bar{\mu}+1)}{4}}^{\frac{\tilde{\lambda}(\bar{\mu}+1)}{4}} {\varphi}_{\bar{\mu},k}(t)\mbox{d}t \right)\left(\int_{\frac{-\tilde{\lambda}(\bar{\mu}+1)}{4}}^{\frac{\tilde{\lambda}(\bar{\mu}+1)}{4}} {\varphi}_{\bar{\mu},k}(t)|W_k(y',0,t,z)|^2\mbox{d}t \right)=\\ \nonumber
&&\left(\int_{\frac{-\tilde{\lambda}(\bar{\mu}+1)}{4}}^{\frac{\tilde{\lambda}(\bar{\mu}+1)}{4}} {\varphi}_{\bar{\mu},k}(t)|W_k(y',0,t,z)|^2\mbox{d}t \right) \ .
\end{eqnarray}
Hence from \eqref{dc} we have
\begin{eqnarray}
&&\int_{B'_{r_0}}|v_k(y',0,z)|^2\mbox{d}y'\le \int_{\frac{-\tilde{\lambda}(\bar{\mu}+1)}{4}}^{\frac{\tilde{\lambda}(\bar{\mu}+1)}{4}}\mbox{d}t \left(  {\varphi}_{\bar{\mu},k}(t)\int_{B'_{r_0}}|W_k(y',0,t,z)|^2 \mbox{d}y'\right)\le \\ \nonumber
&&\left(\int_{\frac{-\tilde{\lambda}(\bar{\mu}+1)}{4}}^{\frac{\tilde{\lambda}(\bar{\mu}+1)}{4}} {\varphi}_{\bar{\mu},k}(t)\mbox{d}t\right) \left(C r_0^{-\frac{1}{2}}\sigma_k \right)^2 \le \left(C r_0^{-\frac{1}{2}}\sigma_k \right)^2 \  .
\end{eqnarray}
\end{proof}

We are now in position to conclude the proof of Theorem \ref{5-115Boundary}. We observe that since the eigenfunctions $e_j$ introduced in \eqref{2-71} are even with respect $y_n$ and since by \eqref{Atildesecondoblocco} we have
\begin{eqnarray}
\tilde{a}_{i,n}(y',0)=0\ \ \mbox{for}\ \ 1\le i\le n-1
\end{eqnarray}
it follows that for any $|y'|\le 2$
\begin{gather}\label{2-93}
\tilde{A}(y',0)\nabla v_k\cdot \nu=- \tilde{a}_{n,n}(y',0)\sum_{j=1}^{\infty} {\alpha}_j{\hat{\varphi}}_{\bar{\mu},k}(\sqrt{\lambda_j})g_k(z\sqrt{\lambda_j})\partial_{y_n}e_j(y',0)=0
\end{gather}
where $\nu=(0,\dots,0,-1)$.
Hence by \eqref{4-5-6-81}, \eqref{1-93} and \eqref{2-93}

\begin{equation}
\label{4-5-6-90}
\left\{\begin{array}{ll}
q(y)\partial^2_{z}v_{k}+\mbox{div}\left(\tilde{A}(y)\nabla_x v_{k}\right)=f_{k}(y,z), \quad |y|\le r_0, |z|\le \frac{r_0}{8},\\[2mm]
\|v_k(\cdot, 0,z)\|_{L^2(B'_{r_0})}\le C r_0^{-\frac{1}{2}}\sigma_k\ \ , \ \ |z|\le \frac{r_0}{8}, \\[2mm]
\tilde{A}(y',0)\nabla v_k\cdot \nu=0 \ , \ \ |y'|\le r_0, \ \ |z|\le \frac{r_0}{8}.
\end{array}\right.
\end{equation}

Finally combining \eqref{3-81}, \eqref{2-81}, quantitative estimates for the Cauchy problem \eqref{4-5-6-90} (see Theorems 3.5 and 3.6 in \cite{Ve3}), we obtain the following
\begin{eqnarray}\label{fine}
\ \ \ \ \ \ \|v_k\|_{L^2(\tilde{B}_{\frac{r_0}{32}})}\le C\left (\epsilon + H_1 (Cr_0)^{2k}\right)^{\beta^2} \left(H_1 e^{2k} + H_1 (Cr_0)^{2k}\right)^{1-\beta^2}
\end{eqnarray}
where $C>0$ depends on $\tilde{\lambda}$ and $\tilde{\Lambda}$\ .

Let us observe that the above inequality replace Theorem 3.6 in \cite{Ve3}.  The same arguments discussed in \cite{Ve3} from Theorem 3.7 and on go through for the present case and lead to the desired estimate \eqref{SUCPBoundary}.

\section {Auxiliary results}\label{AR}

\textbf{Proof of Proposition \ref{solpos}}
Let $\Psi\in C^{1,1}(B_{\rho_0})$ be the map defined as
\begin{equation}
\Psi(y',y_n)=(y',y_n+\phi(y'))
\end{equation}.

For any $r\in (0,\frac{\rho_0}{\sqrt{2}(C+1)})$ we have that
\begin{equation}
K_{\frac{r}{\sqrt{2}(E+1)}}\subset \Psi(B^-_r)\subset K_{{\sqrt{2}(E+1)r}}
\end{equation}
where $B^-_r=\{y\in\mathbb{R}^n\ : |y'|< r\ , y_n < 0 \}$ and furthermore we get
\begin{equation}
|\mbox{det}D \Psi|=1 \ .
\end{equation}

Denoting by
\begin{eqnarray}
&& \sigma(y)=(D\Psi^{-1})(\Psi(y))A(\Psi(y))(D\Psi^{-1})^T(\Psi(y)),\\
 &&\gamma'(y)=\gamma(\Psi(y))\\ 
&&\gamma_0'=\gamma'(0)
\end{eqnarray}
it follows that
\begin{eqnarray}
&&\sigma(0)=A(0)\\
&&\|\sigma_{i,j}\|_{C^{0,1}(B^+{\frac{\rho_0}{\sqrt{2}(C+1)})}}\le \Sigma \ \ , \mbox{for}\ \ i,j=1, \dots, n\\
&&\|\gamma'_{i,j}\|_{C^{0,1}(B'{\frac{\rho_0}{\sqrt{2}(C+1)}(0))}}\le \Lambda'
\end{eqnarray}
where $\Sigma, \Lambda'$ are positive constants depending on $E,\Lambda, \rho_0$\ only.

Dealing as in Proposition 4.3 in \cite{Si} we look for a solution to \eqref{positivesolution} of the form
\begin{eqnarray}
\psi(x',x_n) = \psi'(\Psi^{-1}(x',x_n))
\end{eqnarray}
where $\psi'$ is a solution to
\begin{equation}\label{psiprimo}
\left\{
\begin{array}
{lcl}
\mbox{div}(\sigma(y) \nabla \psi')=0\ ,& \mbox{in $B^{-}_{r_2}$ ,}
\\
\sigma\nabla \psi'\cdot\nu' + \gamma'\psi'=0 \ ,& \mbox{on $B'_{r_2}$ ,}
\end{array}
\right.
\end{equation}
with $r_2=\min\{\rho_0, \frac{\lambda n^{-n/2}}{12\bar{\gamma}} \}$.

And in turn, as in Claim 4.4 of \cite{Si}, we search for a solution $\psi'$ to \eqref{psiprimo} such that $\psi'=\psi_0- s$, where $\psi_0$ is a solution to
\begin{equation}\label{psizero}
\left\{
\begin{array}
{lcl}
\mbox{div}(A(0) \nabla \psi_0)=0\ ,& \mbox{in $B^{-}_{r_2}$ ,}
\\
A(0)\nabla \psi_0\cdot\nu' + \gamma_0'\psi_0=0 \ ,& \mbox{on $B'_{r_2}$ ,}
\end{array}
\right.
\end{equation}
 satisfying $\psi_0\ge 2$ in $B^{-}_{r_2}$ and where $s\in H^1(B^-_{r_2})$ is a weak solution to the problem

\begin{equation}\label{esse}
\left\{
\begin{array}
{lcl}
\mbox{div}(\sigma \nabla s)=-\mbox{div}((\sigma -A(0)) \nabla \psi_0)\ ,& \mbox{in $B^{-}_{r_2}$ ,}
\\
\sigma\nabla s\cdot\nu' + \gamma's=(\sigma - A(0))\nabla \psi_0\cdot \nu' +(\gamma'-\gamma_0)\psi_0 \ ,& \mbox{on $B'_{r_2}$ ,}\\
s=0\ , & \mbox{on $|y|=r_2$}
\end{array}
\right.
\end{equation}
such that $s(y)=\mathcal{O}(|y|^2)$ near the origin. The proof of the latter relies on a slight adaptation of the arguments in Claim 4.4 of \cite{Si}.

In order to construct $\psi_0$, we introduce the following linear change of variable $L=(l_{i,j})_{i,j=1,\dots,n}$ (see also \cite{GaSi})
\begin{eqnarray}
L:&& \mathbb{R}^n \rightarrow \mathbb{R}^n\\
&& \xi \mapsto L \xi = R\sqrt{A^{-1}(0)}\xi
\end{eqnarray}
where $R$ is the planar rotation in $\mathbb{R}^n$ that rotates the unit vector $\frac{v}{\|v\|}$, where $v=\sqrt{A(0)}e_n$ to the $n$th standard unit vector $e_n$, and such that
\[R|_{(\pi)^{\bot}}\equiv Id|_{(\pi)^{\bot}},\]

where $\pi$ is the plane in $\mathbb{R}^n$ generated by $e_n$, $v$ and $(\pi)^{\bot}$ denotes the orthogonal complement of $\pi$ in $\mathbb{R}^{n}$. For this choice of $L$ we have

\begin{description}
\item [i)] $A(0)=L^{-1}\cdot (L^{-1})^T$\ ,
\item [ii)] $(L\xi)\cdot e_n=\frac{1}{||v||} \xi\cdot e_n$.
\end{description}

which means that $L^{-1}:x\mapsto\xi$ is the linear change of variables that maps $I$ into ${A(0)}$.

By defining $\tilde{L}$ as the $(n-1)\times(n-1)$ matrix such that $\tilde{L}=(l)_{i,j=1\cdots,n-1}$ we have that the function
\begin{eqnarray}
\bar{\psi}(\xi)= 8 e^{\displaystyle{-|\mbox{det}L||\mbox{det}\tilde{L}|^{-1}\gamma_0'\xi_n}}\cos(|\mbox{det}L||\mbox{det}\tilde{L}|^{-1}\xi_1\gamma_0')
\end{eqnarray}
is a solution to
\begin{equation}\label{psihat}
\left\{
\begin{array}
{lcl}
\Delta \bar{\psi}=0\ ,& \mbox{in $B^{-}_{r_3}$ ,}
\\
\nabla \bar{\psi}\cdot\nu' + |\mbox{det}L||\mbox{det}\tilde{L}|^{-1}\gamma_0'\bar{\psi}=0\ ,& \mbox{on $B'_{r_3}$ ,}
\end{array}
\right.
\end{equation}
where $r_3=\displaystyle{\frac{1}{2}\frac{\Lambda^{\frac{1}{2}}}{\rho_0}r_2}$\ .

Finally we observe that by setting
\begin{eqnarray}
\psi_0(y)=\bar{\psi}(Ly)
\end{eqnarray}
we end up with a weak solution to \eqref{psizero} such that
\begin{eqnarray}
|\psi_0|>2 \ \  \mbox{in} \ \ \ B^-_{r_2}(0)\ .
\end{eqnarray}

Hence the thesis follows by choosing $r_1=\frac{r_2}{\sqrt{2}(E+1)}$
$\psi(x',x_n) = \psi'(\phi^{-1}(x',x_n))$  and $\psi'=\psi_0- s$.
$\Box$

\begin{prop}\label{autofpari}
There exists a complete orthonormal system of eigenfunctions $e_j$ in $L^2_{+}(B_2, q\mbox{d}y)=\{f\in L^2(B_2,q\mbox{d}y)\ \mbox{s.t.}\ f(y',y_n)=f(y',-y_n)\}$
associated to the Dirichlet problem \eqref{2-70}.
\end{prop}

\begin{proof}
Let us start by observing that from \eqref{Atilde} and since
\begin{subequations}
\label{Atilde2}
\begin{equation}
\label{Atildeinterfaccia}
\tilde a_{ni}(y',0)=\overline{a}_{in}(y',0)=0 ,\quad\hbox{ for}\ \   i\in\{1,\ldots,n-1\},
\end{equation}
\begin{equation}
\label{Atildeorigine}
\tilde a_{nn}(0)=1,
\end{equation}
\end{subequations}
it follows that

\begin{eqnarray}
\mbox{div}(\tilde{A}(y)\nabla_y(u(y',-y_n)))= \mbox{div}(\tilde{A}(z)\nabla_z(u(z)))|_{z=(y',-y_n)}
\end{eqnarray}
for any smooth function $u$.

We set
\begin{eqnarray}
u^{+}(y)=\frac{u(y',y_n) + u(y',-y_n)}{2}
\end{eqnarray}
and we observe that being $q$ even with respect to $y_n$ then we have that if $u$ is a solution to \eqref{1-71} then $u^+$ is a solution to \eqref{1-71} as well.

Let us denote by $\lambda_j$, with $0<\lambda_1\le \lambda_2\le \dots \lambda_j\le \dots $ the eigenvalues associated to the Dirichlet problem \eqref{1-71} and let $\{ \mathrm{S_1}, \mathrm{S_2},\dots,\mathrm{S_j},\dots \}$ be a complete orthonormal system of eigenfunctions in $L^2(B_2, q\mbox{d}y)$\ .

Let us now fix $j\in \mathbb{N}$ and let $\{ \mathrm{S_{j_1}}, \mathrm{S_{j_2}},\dots,\mathrm{S_{j_{k_j}}}\}$ be such that they span the eigenspace corresponding to the eigenvalue $\lambda_j$. We restrict our attention to the non trivial functions $\mathrm{S^+_{j_1}}, \mathrm{S^+_{j_2}},\dots,\mathrm{S^+_{j_{h_j}}}$ among $\mathrm{S^+_{j_1}}, \mathrm{S^+_{j_2}},\dots,\mathrm{S^+_{j_{k_j}}}$ with $h_j\le k_j$.

By a Gram-Schmidt orthogonalization procedure in the Hilbert space $L^2_{+}(B_2, q\mbox{d}y)$ we may find our desired eigenfunctions $e_{j_1}, \dots, e_{j_{h_j}}$ such that
\begin{eqnarray}
(e_{j_l}, e_{j_k})=\int_{B_2} q(y)e_{j_l}(y)e_{j_k}(y)\mbox{d}y=\delta_{{j_l}{j_k}}
\end{eqnarray}
and $e_{j_l}$ are even in $y_n$ for $l=1,\dots, h_j$.

It turns out that the system of eigenfunctions
\begin{eqnarray}
\mathcal{S}=\{e_{1_1}, \dots, e_{1_{h_1}},e_{2_1}, \dots, e_{2_{h_2}}, \dots, e_{j_1}, \dots, e_{j_{h_j}}, \dots \}
\end{eqnarray}
is an orthonormal system by construction. Finally we wish to prove that $\mathcal{S}$ is complete in $L^2_{+}(B_2, q\mbox{d}y)$\ . To this end we assume that  $f\in L^2_{+}(B_2, q\mbox{d}y)$ is such that
\begin{eqnarray}\label{orto}
\int_{B_2}f(y)e(y)q(y)\mbox{d}y=0 \ \ \ \ \forall \ e\in \mathcal{S}
\end{eqnarray}
and we claim that $f\equiv 0$.

In order to prove the claim above, we observe that by \eqref{orto} we have that for any $j\in \mathbb{N}$ the function $f$ in \eqref{orto} is orthogonal with respect the $L^2_{+}(B_2, q\mbox{d}y)$ scalar product to the $\mbox{span}\{e_{j_1}, \dots, e_{j_{h_j}} \}$ and as a consequence to  the $\mbox{span}\{\mathrm{S^+_{j_1}}, \dots, \mathrm{S^+_{j_{k_j}}} \}$ as well.
In particular the following holds
\begin{eqnarray}
\int_{B_2}f(y)q(y)\mathrm{S^+_{j_i}}(y)\mbox{d}y=0\ \ , \ \ j=1,\dots, k_{j}\ .
\end{eqnarray}
On the other hand since $q$ and $f$ are even w.r.t. $y_n$ we have that
\begin{eqnarray}
\int_{B_2}f(y)q(y)\mathrm{S^+_{j_i}}(y)\mbox{d}y=\int_{B_2}f(y)q(y)\mathrm{S_{j_i}}(y)\mbox{d}y\ \ , \ \ j=1,\dots, k_{j}\ .
\end{eqnarray}
Finally we observe that being the system $\{ \mathrm{S_1}, \mathrm{S_2},\dots,\mathrm{S_j},\dots \}$ complete in $L^2(B_2, q\mbox{d}y)$ then $f\equiv 0$ as claimed above.

\end{proof}
\section{Conclusions} \label{conclusions}

Let us conclude by summarizing the main steps of our strategy.
\begin{itemize}
\item We first introduce in Proposition \ref{solpos} a strictly positive solution $\psi$ to the elliptic problem \eqref{positivesolution} such that by the change of variable
\begin{eqnarray}
u^{\star}=\frac{U}{\psi}
\end{eqnarray}
we reformulate our original problem for a Robin boundary condition \eqref{4i-65Boundary}-\eqref{RobinBoundary} in terms of a new one \eqref{4i-65BoundaryN}-\eqref{NeumannBoundary} where a Neumann condition arises instead.

\item Second, in \eqref{2-76} we take advantage of the Boman transform \cite{Bo} in order to perform a suitable transformation of the wave equation in a nonhomogeneous second order elliptic equation \eqref{4-5-6-81}. Furthermore, we observe that the solution $v_k$ to \eqref{4-5-6-81} may be represented as
\begin{eqnarray}\label{v-k2}
v_k(y,z)=\int_{\mathbb{R}}{\varphi}_{\bar{\mu},k}(t)W_k(y,t,z)\mbox{d}t\ \ ,
\end{eqnarray}
where ${\varphi}_{\bar{\mu},k}$ is a suitable sequence of mollifiers and $W_k(y',0,\cdot, \cdot)$ is a solution to the following two dimensional Cauchy problem for a nonhomogeneous elliptic equation
\begin{equation}
\label{Wk2}
\left\{\begin{array}{ll}
\Delta_{t,z}W_k(y',0,t,z)=\partial_t F_k(y'0,t,z) ,\\[2mm]
W_k(y',0,t,0)=\sum_{j=1}^{\infty}\alpha_j \cos(\sqrt{\lambda_j}t)e_j(y',0)=\tilde{u}(y',0,t),\\[2mm]
\partial_{z}W_k(y',0,t,0)=0,
\end{array}\right.
\end{equation}
for any $y\in B_2$\ .

We furthermore, observe that the Dirichlet datum of the above problem can be controlled from above by $\epsilon$ in view of \eqref{CP} and \eqref{4iv-65Boundary}, whereas the Neumann datum vanishes in view of the specific choice discussed in Proposition \ref{autofpari} for the eigenfunctions $e_j$.
The right hand side of the elliptic equation in \eqref{Wk2}, although is in divergence form, it can be handled as well by gathering a refinements of the arguments in Proposition 3.6 of \cite{Ve3} and in Theorem 1.7 of \cite{A-R-R-V}, in order to get the following estimate
\begin{eqnarray}
\int_{B'_{r_0}}|W_k(y',0,t,z)|dy'\le (Cr_0^{\frac{1}{2}}\sigma_k)^2\ .
\end{eqnarray}
\item Finally, by combining the latter with \eqref{v-k} and again the special choice for the eigenfunctions $e_j$ we end up with the Cauchy problem  \eqref{4-5-6-90} which in turn leads to the desired estimate \eqref{fine}.

\end{itemize}
\section*{\normalsize{Acknowledgements}}

This work has been partly supported by INdAM through the project ``Problemi inversi al contorno e sovradeterminati per equazioni alle derivate parziali" GNAMPA 2015.
E. Sincich has been also supported by the grant FRA2014 ``Problemi inversi per PDE, unicit\`{a}, stabilit\`{a}, algoritmi" funded by Universit\`{a} degli Studi di Trieste.

\end{document}